\documentclass[twoside,11pt,reqno]{amsart}
\usepackage{amsmath,amsthm,amssymb,amstext,amsfonts,amscd}
\usepackage{graphicx}
\usepackage{multirow}

\newtheorem{theorem}{Theorem}

\newtheorem{corollary}{Corollary}
\newtheorem{definition}{Definition}

\newtheorem{lemma}{Lemma}

\numberwithin{equation}{section}
\input{tcilatex}

\begin{document}
\title[ Comparative growth analysis of special type of differential
polynomial.....]{Comparative growth analysis of special type of differential
polynomial generated by entire and meromorphic functions on the basis of
their $(p,q)$-th order}
\author[Tanmay Biswas]{Tanmay Biswas}
\address{T. Biswas : Rajbari, Rabindrapalli, R. N. Tagore Road, P.O.-
Krishnagar, Dist-Nadia, PIN-\ 741101, West Bengal, India}
\email{tanmaybiswas\_math@rediffmail.com}

\begin{abstract}
{\small In this paper we aim to establish some results depending on the
comparative growth properties of composite transcendental entire or
meromorphic functions and some special type of differential polynomials
generated by one of the factors on the basis of }$(p,q)${\small -th order
and }$(p,q)${\small -th lower order where }$p,q${\small \ are positive
integers with }$p\geq q${\small .}
\end{abstract}

\maketitle

\section{\textbf{Introduction, Definitions and Notations}}

\qquad Let us consider that the reader is familiar with the fundamental
results and the standard notations of the Nevanlinna theory of meromorphic
functions which are available in \cite{r3, 10, 14, 15}. We also use the
standard notations and definitions of the theory of entire functions which
are available in \cite{16} and therefore we do not explain those in details.
For $x\in \lbrack 0,\infty )$ and $k\in 
\mathbb{N}
$, we define $\exp ^{[k]}x=\exp \left( \exp ^{[k-1]}x\right) $ and $\log
^{[k]}x=\log \left( \log ^{[k-1]}x\right) $ where $%
\mathbb{N}
$ be the set of all positive integers$.$ Let $f$ be an entire function
defined in the open complex plane $%
\mathbb{C}
.$ The\ maximum modulus function $M_{f}\left( r\right) $ corresponding to $f$%
\ is defined on $\left\vert z\right\vert =r$ as $M_{f}\left( r\right) =%
\QTATOP{\max }{\left\vert z\right\vert =r}\left\vert f\left( z\right)
\right\vert $. When $f$ is meromorphic, one may introduce another function $%
T_{f}\left( r\right) $ known as Nevanlinna's characteristic function of $f,$
playing the same role as $M_{f}\left( r\right) .$ However, the Nevanlinna's
Characteristic function of a meromorphic function $f$ is defined as%
\begin{equation*}
T_{f}\left( r\right) =N_{f}\left( r\right) +m_{f}\left( r\right) ,
\end{equation*}%
wherever the function $N_{f}\left( r,a\right) \left( \overset{-}{N_{f}}%
\left( r,a\right) \right) $ known as counting function of\ $a$-points
(distinct $a$-points) of meromorphic $f$ is defined as follows:%
\begin{equation*}
N_{f}\left( r,a\right) =\overset{r}{\underset{0}{\int }}\frac{n_{f}\left(
t,a\right) -n_{f}\left( 0,a\right) }{t}dt+\overset{-}{n_{f}}\left(
0,a\right) \log r
\end{equation*}%
\begin{equation*}
\left( \overset{-}{N_{f}}\left( r,a\right) =\overset{r}{\underset{0}{\int }}%
\frac{\overset{-}{n_{f}}\left( t,a\right) -\overset{-}{n_{f}}\left(
0,a\right) }{t}dt+\overset{-}{n_{f}}\left( 0,a\right) \log r~\right) ,
\end{equation*}%
in addition we represent by $n_{f}\left( r,a\right) \left( \overset{-}{n_{f}}%
\left( r,a\right) \right) $ the number of $a$-points (distinct $a$-points)
of $f$ in $\left\vert z\right\vert \leq r$ and an $\infty $ -point is a pole
of $f$. In many occasions $N_{f}\left( r,\infty \right) $ and $\overset{-}{%
N_{f}}\left( r,\infty \right) $ are symbolized by $N_{f}\left( r\right) $
and $\overset{-}{N_{f}}\left( r\right) $ respectively.

\qquad On the other hand, the function $m_{f}\left( r,\infty \right) $
alternatively indicated by $m_{f}\left( r\right) $ known as the proximity
function of $f$ is defined as:%
\begin{align*}
m_{f}\left( r\right) & =\frac{1}{2\pi }\overset{2\pi }{\underset{0}{\int }}%
\log ^{+}\left\vert f\left( re^{i\theta }\right) \right\vert d\theta ,~\ ~%
\text{where} \\
\log ^{+}x& =\max \left( \log x,0\right) \text{ for all }x\geqslant 0~.
\end{align*}%
Also we may employ $m\left( r,\frac{1}{f-a}\right) $ by $m_{f}\left(
r,a\right) $.

\qquad If $f$ is entire, then the Nevanlinna's Characteristic function $%
T_{f}\left( r\right) $ of $f$ is defined as%
\begin{equation*}
T_{f}\left( r\right) =m_{f}\left( r\right) ~.
\end{equation*}

\qquad Further let\ $n_{0},n_{1},n_{2},.....n_{k}$ \ \ are non negative
integers. For a transcendental meromorphic function $f$, we call the
expression $M[f]=f^{n_{0}}\left( f^{\left( 1\right) }\right) ^{n_{1}}\left(
f^{\left( 2\right) }\right) ^{n_{2}}.......\left( f^{\left( k\right)
}\right) ^{n_{k}}$ to be a monomial generated by\ $f.$ The numbers $\gamma
_{M}=n_{0}+n_{1}+n_{2}+.......+n_{k}$\ and\ $\Gamma
_{M}=n_{0}+2n_{1}+3n_{2}+.......+(k+1)n_{k}$\ are called respectively the
degree and weight of\ the monomial. If $M_{1}\left[ f\right] ,$ $M_{2}\left[
f\right] ,$ $.....,$ $M_{n}[f]$ denote monomials in $f$ , then%
\begin{equation*}
Q[f]=a_{1}M_{1}\left[ f\right] +a_{2}M_{2}[f]+.....+a_{n}M_{n}[f],
\end{equation*}%
$\ $where $a_{i}\neq 0(i=1,2,...,n)$ is called a differential polynomial
generated by $f$ of degree $\gamma _{Q}=\max \{\gamma _{M_{j}}:1\leq j\leq
n\}$ and weight $\Gamma _{Q}=\max \{\Gamma _{M_{J}}:1\leq j\leq n\}.$ Also
we call the numbers$\ \underline{\gamma _{Q}}=\underset{1\leq \text{ }j\leq 
\text{ }s}{\min }\ \gamma _{Mj}$ \ and $k$ (the order of the highest
derivative of\ $f$ ) the lower degree and the order of\ $Q\left[ f\right] $\
respectively. If $\underline{\gamma _{Q}}=\gamma _{Q},$\ $Q\left[ f\right] $%
\ is called a homogeneous differential polynomial.

\qquad However, the ratio $\frac{T_{f}\left( r\right) }{T_{g}\left( r\right) 
}$ as $r\rightarrow \infty $ is called the growth of $f$ with respect to $g$
in terms of the Nevanlinna's Characteristic functions of the meromorphic
functions $f$ and $g$. Moreover, the order $\rho _{f}$ (resp. lower order $%
\lambda _{f}$) of an entire function $f$ which is generally used in
computational purpose is defined as 
\begin{equation*}
\rho _{f}=\overline{\underset{r\rightarrow \infty }{\lim }}\frac{\log \log
M_{f}\left( r\right) }{\log r}\left( \text{resp. }\lambda _{f}=\text{ }%
\underset{r\rightarrow \infty }{\underline{\lim }}\frac{\log \log
M_{f}\left( r\right) }{\log r}\right) ~.
\end{equation*}

\qquad If $f$ is a meromorphic function, then%
\begin{equation*}
\rho _{f}=\overline{\underset{r\rightarrow \infty }{\lim }}\frac{\log
T_{f}\left( r\right) }{\log r}\left( \text{resp. }\lambda _{f}=\underset{%
r\rightarrow \infty }{\underline{\lim }}\frac{\log T_{f}\left( r\right) }{%
\log r}\right) ~\text{.}
\end{equation*}

\qquad Extending this notion, Juneja et. al. \cite{r4} defined the $(p,q)$%
-th order (resp.\ $(p,q)$-th lower order) of an entire function $f$ for any
two positive integers $p,q$ with $p\geq q$ which is as follows:%
\begin{equation*}
\rho _{f^{{}}}\left( p,q\right) =\text{ }\overline{\underset{r\rightarrow
\infty }{\lim }}\frac{\log ^{[p]}M_{f}(r)}{\log ^{\left[ q\right] }r}\text{ }%
\left( \text{resp. }\lambda _{f^{{}}}\left( p,q\right) =\underset{%
r\rightarrow \infty }{\underline{\lim }}\frac{\log ^{[p]}M_{f}(r)}{\log ^{%
\left[ q\right] }r}\right) ~.
\end{equation*}

\qquad If $f$ is meromorphic function, then%
\begin{equation*}
\rho _{f}\left( p,q\right) =\text{ }\overline{\underset{r\rightarrow \infty }%
{\lim }}\frac{\log ^{[p-1]}T_{f}(r)}{\log ^{\left[ q\right] }r}\text{ and }%
\lambda _{f}\left( p,q\right) =\underset{r\rightarrow \infty }{\underline{%
\lim }}\frac{\log ^{[p-1]}T_{f}(r)}{\log ^{\left[ q\right] }r},
\end{equation*}%
where $p,q$ are any two positive integers with $p\geq q$.

\qquad These definitions extend the generalized order $\rho _{f}^{\left[ l%
\right] }$ and generalized lower order $\lambda _{f}^{\left[ l\right] }$ of
an entire function $f$ considered in \cite{r9} for each integer $l\geq 2$
since these correspond to the particular case $\rho _{f}^{\left[ l\right]
}=\rho _{f^{{}}}\left( l,1\right) $ and $\lambda _{f}^{\left[ l\right]
}=\lambda _{f^{{}}}\left( l,1\right) .$ Clearly, $\rho _{f^{{}}}\left(
2,1\right) =\rho _{f}$ and $\lambda _{f^{{}}}\left( 2,1\right) =\lambda
_{f}. $

\qquad An entire or meromorphic function for which $(p,q)$-th order and $%
(p,q)$-th lower order are the same is said to be of regular $\left(
p,q\right) $-growth. Functions which are not of regular $\left( p,q\right) $%
-growth are said to be of irregular $\left( p,q\right) $-growth.

\qquad In this connection we just recall the following two definitions which
will be needed in the sequel.

\begin{definition}
\label{d2} A function $\rho _{f}^{\left[ l\right] }\left( r\right) $ is
called a generalized proximate order of a meromorphic function $f$ relative
to $T_{f}(r)$ if\newline
$\left( i\right) $ $\rho _{f}^{\left[ l\right] }\left( r\right) $ is
non-negative and continuous for $r\geqslant r_{0},$ say,\newline
$\left( ii\right) $ $\rho _{f}^{\left[ l\right] }\left( r\right) $ is
differentiable for $r\geqslant r_{0}$ except possibly at isolated points at
which $\rho _{f}^{\left[ l\right] \prime }\left( r+0\right) $ and $\rho
_{f}^{\left[ l\right] \prime }\left( r-0\right) $ exist,\newline
$\left( iii\right) $ $\underset{r\rightarrow \infty }{\lim }\rho _{f}^{\left[
l\right] }\left( r\right) =\rho _{f}^{\left[ l\right] }<\infty ,$\newline
$\left( iv\right) $ $\underset{r\rightarrow \infty }{\lim }\rho _{f}^{\left[
l\right] \prime }\left( r\right) \overset{l-1}{\underset{i=0}{\Pi }}\log ^{%
\left[ i\right] }r=0$ and\newline
$\left( v\right) $ $\overline{\underset{r\rightarrow \infty }{\lim }}\frac{%
\log ^{\left[ l-2\right] }T_{f}(r)}{r^{\rho _{f}^{\left[ l\right] }\left(
r\right) }}=1.$
\end{definition}

\qquad The existence of such a proximate order is proved by Lahiri \cite{r6}.

\qquad Similarly one can define the generalized lower proximate order of a
meromorphic function $f$ in the following way:

\begin{definition}
\label{d3} A function $\lambda _{f}^{\left[ l\right] }\left( r\right) $ is
defined as a generalized lower proximate order of a meromorphic function $f$
relative to $T_{f}(r)$ if\newline
$\left( i\right) $ $\lambda _{f}^{\left[ l\right] }\left( r\right) $ is
non-negative and continuous for $r\geqslant r_{0},$ say,\newline
$\left( ii\right) $ $\lambda _{f}^{\left[ l\right] }\left( r\right) $ is
differentiable for $r\geqslant r_{0}$ except possibly at isolated points at
which $\lambda _{f}^{\left[ l\right] \prime }\left( r+0\right) $ and $%
\lambda _{f}^{\left[ l\right] \prime }\left( r-0\right) $ exist,\newline
$\left( iii\right) $ $\underset{r\rightarrow \infty }{lim}\lambda _{f}^{%
\left[ l\right] }\left( r\right) =\lambda _{f}^{\left[ l\right] }<\infty ,$%
\newline
$\left( iv\right) $ $\underset{r\rightarrow \infty }{\lim }\lambda _{f}^{%
\left[ l\right] \prime }\left( r\right) \overset{l-1}{\underset{i=0}{\Pi }}%
\log ^{\left[ i\right] }r=0$ and\newline
$\left( v\right) $ $\underset{r\rightarrow \infty }{\underline{\lim }}\frac{%
\log ^{\left[ l-2\right] }T_{f}(r)}{r^{\lambda _{f}^{\left[ l\right] }\left(
r\right) }}=1.$
\end{definition}

\qquad In this paper we aim to establish some results depending on the
comparative growth properties of composite transcendental entire or
meromorphic functions and some special type of differential polynomials
generated by one of the factors on the basis of $(p,q)$-th order ( $(p,q)$%
-th lower order ) and proximate order (proximate lower order) where $p,q$
are positive integers with $p\geq q$.

\section{\textbf{Lemmas}}

\qquad In this section we present some lemmas which will be needed in the
sequel.

\begin{lemma}
\label{l1}\cite{r1} If $f$ is a meromorphic function and $g$ is an entire
function then for all sufficiently large positive numbers of $r,$%
\begin{equation*}
T_{f\circ g}\left( r\right) \leqslant \left\{ 1+o(1)\right\} \frac{%
T_{g}\left( r\right) }{\log M_{g}\left( r\right) }T_{f}\left( M_{g}\left(
r\right) \right) .
\end{equation*}
\end{lemma}

\begin{lemma}
\label{l7}\cite{r2} Suppose that $f$ is a meromorphic function and $g$ be an
entire function and suppose that $0<\mu <\rho _{g}\leq \infty .$Then for a
sequence of values of $r$ tending to infinity,%
\begin{equation*}
T_{f\circ g}(r)\geq T_{f}\left( \exp \left( r^{\mu }\right) \right) ~.
\end{equation*}
\end{lemma}

\begin{lemma}
\label{l5}\cite{xx} Let $g$ be an entire function. Then for any $\delta (>0)$
the function $r^{\lambda _{g}^{\left[ l\right] }+\delta -\lambda _{g}^{\left[
l\right] }\left( r\right) }$ is ultimately an increasing function of $r.$
\end{lemma}

\begin{lemma}
\label{l6}\cite{xx} Let $g$ be an entire function. Then for any $\delta (>0)$
the function $r^{\rho _{g}^{\left[ l\right] }+\delta -\rho _{g}^{\left[ l%
\right] }\left( r\right) }$ is ultimately an increasing function of $r.$
\end{lemma}

\begin{lemma}
\label{l8}\cite{y} Let $f$ be a transcendental meromorphic function and $%
F=f^{n}Q\left[ f\right] $ where $Q\left[ f\right] $ is a differential
polynomial in $f$, then for any $n\geq 1$%
\begin{eqnarray*}
T_{f}\left( r\right) &=&O\left\{ T_{F}\left( r\right) \right\} \text{ as }%
r\rightarrow \infty \\
\text{and }T_{F}\left( r\right) &=&O\left\{ T_{f}\left( r\right) \right\} 
\text{ as }r\rightarrow \infty ~.
\end{eqnarray*}
\end{lemma}

\begin{lemma}
\label{l9} Let $f$ be a transcendental meromorphic function and $F=f^{n}Q%
\left[ f\right] $ where $Q\left[ f\right] $ is a differential polynomial in $%
f$, then for any $n\geq 1$%
\begin{equation*}
\rho _{F}\left( p,q\right) =\text{ }\rho _{f}\left( p,q\right) \text{ and }%
\lambda _{F}\left( p,q\right) =\lambda _{f}\left( p,q\right) ~.
\end{equation*}
\end{lemma}

\begin{proof}
Let us consider that $\alpha $ and $\beta $ be any two constant greater than 
$1$. Now we get from Lemma \ref{l8} for all sufficiently large values of $r$
that%
\begin{equation}
T_{F}(r)<\alpha \cdot T_{f}\left( r\right)  \label{1.}
\end{equation}%
and%
\begin{equation}
T_{f}\left( r\right) <\beta \cdot T_{F}(r)~.  \label{2.}
\end{equation}

\qquad Now from $\left( \ref{1.}\right) $ it follows for all sufficiently
large values of $r$ that%
\begin{eqnarray}
\log ^{\left[ p\right] }T_{F}(r) &<&\log ^{\left[ p\right] }T_{f}\left(
r\right) +O(1)  \notag \\
i.e.,~\frac{\log ^{\left[ p\right] }T_{F}(r)}{\log ^{\left[ q\right] }r} &<&%
\frac{\log ^{\left[ p\right] }T_{f}\left( r\right) +O(1)}{\log ^{\left[ q%
\right] }r}  \notag \\
i.e.,~\rho _{F}\left( p,q\right) &\leq &\text{ }\rho _{f}\left( p,q\right) ~.
\label{3.}
\end{eqnarray}

\qquad Again from $\left( \ref{2.}\right) $ we obtain for all sufficiently
large values of $r$ that%
\begin{eqnarray}
\log ^{\left[ p\right] }T_{f}\left( r\right) &<&\log ^{\left[ p\right]
}T_{F}(r)+O(1)  \notag \\
i.e.,~\frac{\log ^{\left[ p\right] }T_{f}\left( r\right) }{\log ^{\left[ q%
\right] }r} &<&\frac{\log ^{\left[ p\right] }T_{F}(r)+O(1)}{\log ^{\left[ q%
\right] }r}  \notag \\
i.e.,~\rho _{f}\left( p,q\right) &\leq &\text{ }\rho _{F}\left( p,q\right) ~.
\label{4.}
\end{eqnarray}

\qquad Therefore from $\left( \ref{3.}\right) $ and $\left( \ref{4.}\right)
, $ we get that%
\begin{equation*}
\rho _{F}\left( p,q\right) =\text{ }\rho _{f}\left( p,q\right) ~.
\end{equation*}

\qquad In a similar manner, $\lambda _{F}\left( p,q\right) =\lambda
_{f}\left( p,q\right) .$

\qquad Thus the lemma follows.
\end{proof}

\section{\textbf{Main Results}}

\qquad In this section we present the main results of the paper.

\begin{theorem}
\label{t1} Let $f$\ be a transcendental meromorphic function and $g$ be an
entire function such that $\rho _{g}(m,n)<\lambda _{f}(p,q)$ $\leq \rho
_{f}(p,q)<\infty $ where $p,q,m,n$ are positive integers with $p\geq q,m\geq
n$. Also let $F=f^{\alpha }Q\left[ f\right] $ where $Q\left[ f\right] $ is a
differential polynomial in $f$, then for any $\alpha \geq 1$%
\begin{equation*}
\left( i\right) \underset{r\rightarrow \infty }{\lim }\frac{\log ^{\left[ p-1%
\right] }T_{f\circ g}\left( \exp ^{\left[ n-1\right] }r\right) }{\log ^{%
\left[ p-2\right] }T_{F}(\exp ^{\left[ q-1\right] }r)}=0\text{ if }q\geq m~\
\ \ \ \ \ \ 
\end{equation*}%
and%
\begin{equation*}
\left( ii\right) \underset{r\rightarrow \infty }{\lim }\frac{\log ^{\left[
p+m-q-2\right] }T_{f\circ g}\left( \exp ^{\left[ n-1\right] }r\right) }{\log
^{\left[ p-2\right] }T_{F}(\exp ^{\left[ q-1\right] }r)}=0\text{ if }q<m.
\end{equation*}
\end{theorem}

\begin{proof}
Since $\rho _{g}(m,n)<\lambda _{f}(p,q)$ we can choose $\varepsilon \left(
>0\right) $ is such a way that%
\begin{equation}
\rho _{g}(m,n)+\varepsilon <\lambda _{f}(p,q)-\varepsilon .  \label{1x}
\end{equation}%
As $T_{g}(r)\leq \log ^{+}M_{g}(r)$ \{cf. \cite{r3} \}$,$ we have from Lemma %
\ref{l1}, for all sufficiently large values of $r$ that%
\begin{equation*}
\log ^{\left[ p-1\right] }T_{f\circ g}\left( \exp ^{\left[ n-1\right]
}r\right) \leq \log ^{\left[ p-1\right] }T_{f}\left( M_{g}\left( \exp ^{%
\left[ n-1\right] }r\right) \right) +O(1)
\end{equation*}%
\begin{equation}
i.e.,~\log ^{\left[ p-1\right] }T_{f\circ g}\left( \exp ^{\left[ n-1\right]
}r\right) \leqslant \left( \rho _{f}\left( p,q\right) +\varepsilon \right)
\log ^{\left[ q\right] }M_{g}\left( \exp ^{\left[ n-1\right] }r\right) +O(1).
\label{2x}
\end{equation}%
Now the following two cases may arise .\newline
\textbf{Case I.} Let $q\geqslant m$. Then we have from $\left( \ref{2x}%
\right) $ for all sufficiently large values of $r$ that%
\begin{equation}
\log ^{\left[ p-1\right] }T_{f\circ g}\left( \exp ^{\left[ n-1\right]
}r\right) \leqslant \left( \rho _{f}\left( p,q\right) +\varepsilon \right)
\log ^{\left[ m-1\right] }M_{g}\left( \exp ^{\left[ n-1\right] }r\right)
+O(1).  \label{3x}
\end{equation}%
Again for all sufficiently large values of $r,$%
\begin{eqnarray}
\log ^{\left[ m\right] }M_{g}\left( \exp ^{\left[ n-1\right] }r\right)
&\leqslant &\left( \rho _{g}(m,n)+\varepsilon \right) \log ^{\left[ n\right]
}\exp ^{\left[ n-1\right] }r  \notag \\
i.e.,~\log ^{\left[ m-1\right] }M_{g}\left( \exp ^{\left[ n-1\right]
}r\right) &\leqslant &r^{\left( \rho _{g}(m,n)+\varepsilon \right) }.
\label{4x}
\end{eqnarray}%
Now from $\left( \ref{3x}\right) $ and $\left( \ref{4x}\right) $ we have for
all sufficiently large values of $r$ that%
\begin{equation}
\log ^{\left[ p-1\right] }T_{f\circ g}\left( \exp ^{\left[ n-1\right]
}r\right) \leqslant \left( \rho _{f}\left( p,q\right) +\varepsilon \right)
r^{\left( \rho _{g}(m,n)+\varepsilon \right) }+O(1).  \label{5x}
\end{equation}%
\textbf{Case II.} Let $q<m.$ Then for all sufficiently large values of $r$
we get from $\left( \ref{2x}\right) $ that%
\begin{equation}
\log ^{\left[ p-1\right] }T_{f\circ g}\left( \exp ^{\left[ n-1\right]
}r\right) \leqslant \left( \rho _{f}\left( p,q\right) +\varepsilon \right)
\exp ^{\left[ m-q\right] }\log ^{\left[ m\right] }M_{g}\left( \exp ^{\left[
n-1\right] }r\right) +O(1).  \label{6x}
\end{equation}%
Again for all sufficiently large values of $r,$%
\begin{eqnarray}
\log ^{\left[ m\right] }M_{g}\left( \exp ^{\left[ n-1\right] }r\right)
&\leqslant &\left( \rho _{g}(m,n)+\varepsilon \right) \log ^{\left[ n\right]
}\exp ^{\left[ n-1\right] }r  \notag \\
i.e.,~\log ^{\left[ m\right] }M_{g}\left( \exp ^{\left[ n-1\right] }r\right)
&\leqslant &\log r^{\rho _{g}(m,n)+\varepsilon }  \notag \\
i.e.,~\exp ^{\left[ m-q\right] }\log ^{\left[ m\right] }M_{g}\left( \exp ^{%
\left[ n-1\right] }r\right) &\leqslant &\exp ^{\left[ m-q\right] }\log
r^{\rho _{g}(m,n)+\varepsilon }  \notag
\end{eqnarray}%
\begin{equation}
i.e.,~\exp ^{\left[ m-q\right] }\log ^{\left[ m\right] }M_{g}\left( \exp ^{%
\left[ n-1\right] }r\right) \leqslant \exp ^{\left[ m-q-1\right] }r^{\rho
_{g}(m,n)+\varepsilon }.  \label{7x}
\end{equation}%
Now from $\left( \ref{6x}\right) $ and $\left( \ref{7x}\right) $ we have for
all sufficiently large values of $r$ that%
\begin{equation*}
\log ^{\left[ p-1\right] }T_{f\circ g}\left( \exp ^{\left[ n-1\right]
}r\right) \leq \left( \rho _{f}\left( p,q\right) +\varepsilon \right) \exp ^{%
\left[ m-q-1\right] }r^{\rho _{g}(m,n)+\varepsilon }+O(1)
\end{equation*}%
\begin{eqnarray}
i.e.,~\log ^{\left[ p\right] }T_{f\circ g}\left( \exp ^{\left[ n-1\right]
}r\right) &\leqslant &\exp ^{\left[ m-q-2\right] }r^{\rho
_{g}(m,n)+\varepsilon }+O(1)  \notag \\
i.e.,~\log ^{\left[ p+m-q-2\right] }T_{f\circ g}\left( \exp ^{\left[ n-1%
\right] }r\right) &\leqslant &\log ^{[m-q-2]}\exp ^{\left[ m-q-2\right]
}r^{\rho _{g}(m,n)+\varepsilon }+O(1)  \notag
\end{eqnarray}%
\begin{equation}
i.e.,~\log ^{\left[ p+m-q-2\right] }T_{f\circ g}\left( \exp ^{\left[ n-1%
\right] }r\right) \leqslant r^{\rho _{g}(m,n)+\varepsilon }+O(1)~.
\label{8x}
\end{equation}%
Again for all sufficiently large values of $r,$ we get in view of Lemma \ref%
{l9} that%
\begin{eqnarray}
\log ^{\left[ p-1\right] }T_{F}(\exp ^{\left[ q-1\right] }r) &\geqslant
&(\lambda _{F}(p,q)-\varepsilon )\log ^{\left[ q\right] }\exp ^{\left[ q-1%
\right] }r  \notag \\
i.e.,~\log ^{\left[ p-1\right] }T_{F}(\exp ^{\left[ q-1\right] }r)
&\geqslant &(\lambda _{f}(p,q)-\varepsilon )\log r  \notag \\
i.e.,~\log ^{\left[ p-1\right] }T_{F}(\exp ^{\left[ q-1\right] }r)
&\geqslant &\log r^{(\lambda _{f}(p,q)-\varepsilon )}  \notag \\
i.e.,~\log ^{\left[ p-2\right] }T_{F}(\exp ^{\left[ q-1\right] }r)
&\geqslant &r^{(\lambda _{f}(p,q)-\varepsilon )}~.  \label{9x}
\end{eqnarray}%
Now combining $\left( \ref{5x}\right) $ of Case I and $\left( \ref{9x}%
\right) $ we get for all sufficiently large values of $r$ that%
\begin{equation}
\frac{\log ^{\left[ p-1\right] }T_{f\circ g}\left( \exp ^{\left[ n-1\right]
}r\right) }{\log ^{\left[ p-2\right] }T_{F}(\exp ^{\left[ q-1\right] }r)}%
\leq \frac{\left( \rho _{f}\left( p,q\right) +\varepsilon \right) r^{\left(
\rho _{g}(m,n)+\varepsilon \right) }+O(1)}{r^{(\lambda _{f}(p,q)-\varepsilon
)}}.  \label{10x}
\end{equation}%
Now in view of $\left( \ref{1x}\right) $ it follows from $\left( \ref{10x}%
\right) $ that%
\begin{equation*}
\underset{r\rightarrow \infty }{\lim }\frac{\log ^{\left[ p-1\right]
}T_{f\circ g}\left( \exp ^{\left[ n-1\right] }r\right) }{\log ^{\left[ p-2%
\right] }T_{F}(\exp ^{\left[ q-1\right] }r)}=0~.
\end{equation*}%
This proves the first part of the theorem.

Again combining $\left( \ref{8x}\right) $ of Case II and $\left( \ref{9x}%
\right) $ we obtain for all sufficiently large values of $r$ that%
\begin{equation}
\frac{\log ^{\left[ p+m-q-2\right] }T_{f\circ g}\left( \exp ^{\left[ n-1%
\right] }r\right) }{\log ^{\left[ p-2\right] }T_{F}(\exp ^{\left[ q-1\right]
}r)}\leq \frac{r^{\rho _{g}(m,n)+\varepsilon }+O(1)}{r^{(\lambda
_{f}(p,q)-\varepsilon )}}.  \label{11x}
\end{equation}%
Now in view of $\left( \ref{1x}\right) $ it follows from $\left( \ref{11x}%
\right) $ that%
\begin{equation*}
\underset{r\rightarrow \infty }{\lim }\frac{\log ^{\left[ p+m-q-2\right]
}T_{f\circ g}\left( \exp ^{\left[ n-1\right] }r\right) }{\log ^{\left[ p-2%
\right] }T_{F}(\exp ^{\left[ q-1\right] }r)}=0~.
\end{equation*}%
Thus the theorem follows.
\end{proof}

\begin{theorem}
\label{t2} Let $f$\ be a transcendental meromorphic function and $g$ be an
entire function such that $\lambda _{g}(m,n)<\lambda _{f}(p,q)\leq \rho
_{f}(p,q)<\infty $ where $p,q,m,n$ are positive integers with $p\geq q$ and $%
m\geq n$. Also let $F=f^{\alpha }Q\left[ f\right] $ where $Q\left[ f\right] $
is a differential polynomial in $f$, then for any $\alpha \geq 1$%
\begin{equation*}
\left( i\right) \underset{r\rightarrow \infty }{\underline{\lim }}\frac{\log
^{\left[ p-1\right] }T_{f\circ g}\left( \exp ^{\left[ n-1\right] }r\right) }{%
\log ^{\left[ p-2\right] }T_{F}(\exp ^{\left[ q-1\right] }r)}=0\text{ if }%
q\geq m~\ \ \ \ \ \ \ 
\end{equation*}%
and%
\begin{equation*}
\left( ii\right) \underset{r\rightarrow \infty }{\underline{\lim }}\frac{%
\log ^{\left[ p+m-q-2\right] }T_{f\circ g}\left( \exp ^{\left[ n-1\right]
}r\right) }{\log ^{\left[ p-2\right] }T_{F}(\exp ^{\left[ q-1\right] }r)}=0%
\text{ if }q<m.
\end{equation*}
\end{theorem}

\begin{proof}
For a sequence of values of $r$ tending to infinity that%
\begin{eqnarray}
\log ^{\left[ m\right] }M_{g}\left( \exp ^{\left[ n-1\right] }r\right)
&\leqslant &\left( \lambda _{g}(m,n)+\varepsilon \right) \log ^{\left[ n%
\right] }\exp ^{\left[ n-1\right] }r  \notag \\
i.e.,~\log ^{\left[ m\right] }M_{g}\left( \exp ^{\left[ n-1\right] }r\right)
&\leqslant &\log r^{\lambda _{g}(m,n)+\varepsilon }  \notag \\
i.e.,~\log ^{\left[ m-1\right] }M_{g}\left( \exp ^{\left[ n-1\right]
}r\right) &\leqslant &\log r^{\lambda _{g}(m,n)+\varepsilon }~.  \label{12x}
\end{eqnarray}%
Now from $\left( \ref{3x}\right) $ and $\left( \ref{12x}\right) $ we get for
a sequence of values of $r$ tending to infinity that%
\begin{equation}
\log ^{\left[ p-1\right] }T_{f\circ g}\left( \exp ^{\left[ n-1\right]
}r\right) \leqslant \left( \rho _{f}\left( p,q\right) +\varepsilon \right)
r^{\lambda _{g}(m,n)+\varepsilon }+O(1).  \label{13x}
\end{equation}%
Combining $\left( \ref{9x}\right) $ and $\left( \ref{13x}\right) $ we obtain
for a sequence of values of $r$ tending to infinity that%
\begin{equation}
\frac{\log ^{\left[ p-1\right] }T_{f\circ g}\left( \exp ^{\left[ n-1\right]
}r\right) }{\log ^{\left[ p-2\right] }T_{F}(\exp ^{\left[ q-1\right] }r)}%
\leq \frac{\left( \rho _{f}\left( p,q\right) +\varepsilon \right) r^{\lambda
_{g}(m,n)+\varepsilon }+O(1)}{r^{(\lambda _{f}(p,q)-\varepsilon )}}.
\label{14x}
\end{equation}%
Now in view of $\left( \ref{1x}\right) $ we have from $\left( \ref{14x}%
\right) $ that%
\begin{equation*}
\underset{r\rightarrow \infty }{\underline{\lim }}\frac{\log ^{\left[ p-1%
\right] }T_{f\circ g}\left( \exp ^{\left[ n-1\right] }r\right) }{\log ^{%
\left[ p-2\right] }T_{F}(\exp ^{\left[ q-1\right] }r)}=0~.
\end{equation*}%
This proves the first part of the theorem.

Again for a sequence of values of $r$ tending to infinity that%
\begin{eqnarray}
\log ^{\left[ m\right] }M_{g}\left( \exp ^{\left[ n-1\right] }r\right)
&\leqslant &\left( \lambda _{g}(m,n)+\varepsilon \right) \log ^{\left[ n%
\right] }\exp ^{\left[ n-1\right] }r  \notag \\
i.e.,~\log ^{\left[ m\right] }M_{g}\left( \exp ^{\left[ n-1\right] }r\right)
&\leqslant &\log r^{\left( \lambda _{g}(m,n)+\varepsilon \right) }  \notag \\
i.e.,~\exp ^{\left[ m-q\right] }\log ^{\left[ m\right] }M_{g}\left( \exp ^{%
\left[ n-1\right] }r\right) &\leqslant &\exp ^{\left[ m-q\right] }\log
r^{\left( \lambda _{g}(m,n)+\varepsilon \right) }  \notag \\
i.e.,~\exp ^{\left[ m-q\right] }\log ^{\left[ m\right] }M_{g}\left( \exp ^{%
\left[ n-1\right] }r\right) &\leqslant &\exp ^{\left[ m-q-1\right]
}r^{\left( \lambda _{g}(m,n)+\varepsilon \right) }.  \label{15x}
\end{eqnarray}%
Now from $\left( \ref{6x}\right) $ and $\left( \ref{15x}\right) $ we have
for a sequence of values of $r$ tending to infinity that%
\begin{eqnarray*}
\log ^{\left[ p-1\right] }T_{f\circ g}\left( \exp ^{\left[ n-1\right]
}r\right) &\leqslant &\left( \rho _{f}\left( p,q\right) +\varepsilon \right)
\exp ^{\left[ m-q-1\right] }r^{\left( \lambda _{g}(m,n)+\varepsilon \right)
}+O(1) \\
i.e.,~\log ^{\left[ p\right] }T_{f\circ g}\left( \exp ^{\left[ n-1\right]
}r\right) &\leqslant &\exp ^{\left[ m-q-2\right] }r^{\left( \lambda
_{g}(m,n)+\varepsilon \right) }+O(1)
\end{eqnarray*}%
\begin{equation*}
i.e.,~\log ^{\left[ p+m-q-2\right] }T_{f\circ g}\left( \exp ^{\left[ n-1%
\right] }r\right) \leqslant \log ^{[m-q-2]}\exp ^{\left[ m-q-2\right]
}r^{\left( \lambda _{g}(m,n)+\varepsilon \right) }+O(1)
\end{equation*}%
\begin{equation}
i.e.,~\log ^{\left[ p+m-q-2\right] }T_{f\circ g}\left( \exp ^{\left[ n-1%
\right] }r\right) \leqslant r^{\left( \lambda _{g}(m,n)+\varepsilon \right)
}+O(1).~\ \ \ \ \ \ \ \ \ \ \ \ \ \ \ \ \ \ \ \ \ \ \ \   \label{16x}
\end{equation}%
Combining $\left( \ref{9x}\right) $ and $\left( \ref{16x}\right) $ we obtain
for a sequence of values of $r$ tending to infinity that%
\begin{equation}
\frac{\log ^{\left[ p+m-q-2\right] }T_{f\circ g}\left( \exp ^{\left[ n-1%
\right] }r\right) }{\log ^{\left[ p-2\right] }T_{F}(\exp ^{\left[ q-1\right]
}r)}\leq \frac{r^{\left( \lambda _{g}(m,n)+\varepsilon \right) }+O(1)}{%
r^{\lambda _{f}(p,q)-\varepsilon }}.  \label{17x}
\end{equation}%
Now in view of $\left( \ref{1x}\right) $ it follows from $\left( \ref{17x}%
\right) $ that%
\begin{equation*}
\underset{r\rightarrow \infty }{\underline{\lim }}\frac{\log ^{\left[ p+m-q-2%
\right] }T_{f\circ g}\left( \exp ^{\left[ n-1\right] }r\right) }{\log ^{%
\left[ p-2\right] }T_{F}(\exp ^{\left[ q-1\right] }r)}=0.
\end{equation*}%
This establishes the second part of the theorem.
\end{proof}

\begin{theorem}
\label{t3} Let $g$ be an entire function and $f$\ be a transcendental
meromorphic function such that $0<\lambda _{f}(p,q)\leq \rho
_{f}(p,q)<\infty $ where $p$ and $q$ are any two positive integers with $%
p\geq q$. Also let $F=f^{\alpha }Q\left[ f\right] $ where $Q\left[ f\right] $
is a differential polynomial in $f$, then for any $\alpha \geq 1$%
\begin{equation*}
(i)~\overline{\underset{r\rightarrow \infty }{\lim }}\frac{\log ^{\left[ p-1%
\right] }T_{f\circ g}(r)}{\log ^{\left[ p-1\right] }T_{F}(\exp \left( r^{\mu
}\right) )}=\infty \text{ if }q=1~\ \ \ \ \ \ \ \ \ \ \ \ \ 
\end{equation*}%
\begin{equation*}
(ii)~\overline{\underset{r\rightarrow \infty }{\lim }}\frac{\log ^{\left[ p-1%
\right] }T_{f\circ g}(r)}{\log ^{\left[ p-1\right] }T_{F}(\exp \left( r^{\mu
}\right) )}\geq \frac{\beta \lambda _{f}\left( p,q\right) }{\mu \rho
_{f}\left( p,q\right) }\text{ if }q=2
\end{equation*}%
and%
\begin{equation*}
(iii)~\overline{\underset{r\rightarrow \infty }{\lim }}\frac{\log ^{\left[
p-1\right] }T_{f\circ g}(r)}{\log ^{\left[ p-1\right] }T_{F}(\exp \left(
r^{\mu }\right) )}\geq \frac{\lambda _{f}\left( p,q\right) }{\rho _{f}\left(
p,q\right) }\text{ if }q>2~\ 
\end{equation*}%
where $0<\mu <\beta <\rho _{g}$ .
\end{theorem}

\begin{proof}
Since $0<\mu <\beta <\rho _{g},$ then from Lemma \ref{l7} we obtain for a
sequence of values of $r$ tending to infinity that%
\begin{eqnarray}
\log ^{\left[ p-1\right] }T_{f\circ g}(r) &\geq &\log ^{\left[ p-1\right]
}T_{f}\left( \exp \left( r^{\alpha }\right) \right)  \notag \\
i.e.,~\log ^{\left[ p-1\right] }T_{f\circ g}(r) &\geq &(\lambda _{f}\left(
p,q\right) -\varepsilon )\log ^{\left[ q\right] }\exp \left( r^{\alpha
}\right)  \notag \\
i.e.,~\log ^{\left[ p-1\right] }T_{f\circ g}(r) &\geq &(\lambda _{f}\left(
p,q\right) -\varepsilon )\log ^{\left[ q-1\right] }\left( r^{\alpha }\right)
~.  \label{20x}
\end{eqnarray}%
Again from the definition of $\rho _{F}\left( p,q\right) $ it follows in
view of Lemma \ref{l9}, for all sufficiently large values of $r$ that%
\begin{equation*}
\log ^{\left[ p-1\right] }T_{F}(\exp \left( r^{\mu }\right) )\leq \left(
\rho _{F}\left( p,q\right) +\varepsilon \right) \log ^{\left[ q\right] }\exp
\left( r^{\mu }\right)
\end{equation*}%
\begin{equation}
i.e.,~\log ^{\left[ p-1\right] }T_{F}(\exp \left( r^{\mu }\right) )\leq
\left( \rho _{f}\left( p,q\right) +\varepsilon \right) \log ^{\left[ q-1%
\right] }\left( r^{\mu }\right) ~.  \label{21x}
\end{equation}%
Thus from $\left( \ref{20x}\right) $ and $\left( \ref{21x}\right) $ we have
for a sequence of values of $r$ tending to infinity that%
\begin{equation}
\frac{\log ^{\left[ p-1\right] }T_{f\circ g}(r)}{\log ^{\left[ p-1\right]
}T_{F}(\exp \left( r^{\mu }\right) )}\geq \frac{(\lambda _{f}\left(
p,q\right) -\varepsilon )\log ^{\left[ q-1\right] }\left( r^{\alpha }\right) 
}{\left( \rho _{f}\left( p,q\right) +\varepsilon \right) \log ^{\left[ q-1%
\right] }\left( r^{\mu }\right) }~.  \label{22x}
\end{equation}%
Since $\mu <\beta $, the theorem follows from $\left( \ref{22x}\right) .$
\end{proof}

\begin{theorem}
\label{t4} Let $f$\ be a transcendental meromorphic function and $g$ be an
entire function such that $0<\lambda _{f}(p,q)\leq \rho _{f}(p,q)<\infty $
and $\rho _{g}(m,n)$ $<$ $\infty $ where $p,q,m,n$ are positive integers
with $p\geq q$ and $m\geq n$. Also let $F=f^{\alpha }Q\left[ f\right] $
where $Q\left[ f\right] $ is a differential polynomial in $f$, then for any $%
\alpha \geq 1$%
\begin{equation*}
(i)~\overline{\underset{r\rightarrow \infty }{\lim }}\frac{\log ^{\left[ p%
\right] }T_{f\circ g}\left( \exp ^{\left[ n-1\right] }r\right) }{\log ^{%
\left[ p-1\right] }T_{F}(\exp ^{\left[ q-1\right] }r)}\leqslant \frac{\rho
_{g}(m,n)}{\lambda _{f}(p,q)}~\text{if }q\geq m~\ \ \ \ \ \ \ \ \ \ \ 
\end{equation*}%
and%
\begin{equation*}
(ii)~\overline{\underset{r\rightarrow \infty }{\lim }}\frac{\log ^{\left[
p+m-q-1\right] }T_{f\circ g}\left( \exp ^{\left[ n-1\right] }r\right) }{\log
^{\left[ p-1\right] }T_{F}(\exp ^{\left[ q-1\right] }r)}\leqslant \frac{\rho
_{g}(m,n)}{\lambda _{f}(p,q)}~\text{if }q<m~.
\end{equation*}
\end{theorem}

\begin{proof}
In view of Lemma \ref{l9}, we have for all sufficiently large values of $r$
that%
\begin{eqnarray}
\log ^{\left[ p-1\right] }T_{F}(\exp ^{\left[ q-1\right] }r) &\geq &\left(
\lambda _{F}(p,q)-\varepsilon \right) \log ^{\left[ q\right] }\exp ^{\left[
q-1\right] }r  \notag \\
i.e.,~\log ^{\left[ p-1\right] }T_{F}(\exp ^{\left[ q-1\right] }r) &\geq
&\left( \lambda _{f}(p,q)-\varepsilon \right) \log r~.  \label{23x}
\end{eqnarray}%
\textbf{Case I.} If $q\geqslant m$ , then from $\left( \ref{5x}\right) $ and 
$\left( \ref{23x}\right) $ we get for all sufficiently large values of $r$
that

\begin{equation*}
\frac{\log ^{\left[ p\right] }T_{f\circ g}\left( \exp ^{\left[ n-1\right]
}r\right) }{\log ^{\left[ p-1\right] }T_{F}(\exp ^{\left[ q-1\right] }r)}%
\leqslant \frac{\left( \rho _{g}(m,n)+\varepsilon \right) \log r+O(1)}{%
\left( \lambda _{f}(p,q)-\varepsilon \right) \log r}~.
\end{equation*}%
Since $\varepsilon \left( >0\right) $ is arbitrary, it follows from above
that%
\begin{equation*}
\overline{\underset{r\rightarrow \infty }{\lim }}\frac{\log ^{\left[ p\right]
}T_{f\circ g}\left( \exp ^{\left[ n-1\right] }r\right) }{\log ^{\left[ p-1%
\right] }T_{F}(\exp ^{\left[ q-1\right] }r)}\leqslant \frac{\rho _{g}(m,n)}{%
\lambda _{f}(p,q)}~.
\end{equation*}%
This proves the first part of the theorem.$\newline
$\textbf{Case II.} If $q<m$ then from $\left( \ref{8x}\right) $ and $\left( %
\ref{23x}\right) $ we obtain for all sufficiently large values of $r$ that%
\begin{equation*}
\frac{\log ^{\left[ p+m-q-1\right] }T_{f\circ g}\left( \exp ^{\left[ n-1%
\right] }r\right) }{\log ^{\left[ p-1\right] }T_{F}(\exp ^{\left[ q-1\right]
}r)}\leqslant \frac{\left( \rho _{g}(m,n)+\varepsilon \right) \log r+O(1)}{%
\left( \lambda _{f}(p,q)-\varepsilon \right) \log r}~.
\end{equation*}%
As $\varepsilon \left( >0\right) $ is arbitrary, it follows from above that%
\begin{equation*}
\overline{\underset{r\rightarrow \infty }{\lim }}\frac{\log ^{\left[ p+m-q-1%
\right] }T_{f\circ g}\left( \exp ^{\left[ n-1\right] }r\right) }{\log ^{%
\left[ p-1\right] }T_{F}(\exp ^{\left[ q-1\right] }r)}\leqslant \frac{\rho
_{g}(m,n)}{\lambda _{f}(p,q)}~.
\end{equation*}%
Thus the second part of the theorem is established.
\end{proof}

\begin{theorem}
\label{t5} If $f$\ be meromorphic and $g$ be a transcendental entire such
that $\rho _{f}(p,q)$ and $\lambda _{g}^{\left[ l\right] }$ are both finite
where $p,q,l$ are positive integers with $p>q$ and $l\geq 2$. Also let $%
G=g^{\alpha }Q\left[ g\right] $ where $Q\left[ g\right] $ is a differential
polynomial in $g$, then for any $\alpha \geq 1$%
\begin{eqnarray*}
\left( i\right) ~~\ \ \ \ ~\underset{r\rightarrow \infty }{\underline{\lim }}%
\frac{\log ^{\left[ p-1\right] }T_{f\circ g}\left( r\right) }{\log ^{\left[
l-2\right] }T_{G}(r)} &\leqslant &\rho _{f}(p,q).2^{^{\lambda _{g}^{\left[ l%
\right] }}}\text{ if }q\geq l-1>1~, \\
\left( ii\right) ~\underset{r\rightarrow \infty }{\underline{\lim }}\frac{%
\log ^{\left[ p+l-q-2\right] }T_{f\circ g}\left( r\right) }{\log ^{\left[ l-2%
\right] }T_{G}(r)} &\leqslant &2^{^{\lambda _{g}^{\left[ l\right] }}}\text{
if }q<l-1,
\end{eqnarray*}%
and%
\begin{equation*}
\left( iii\right) ~\underset{r\rightarrow \infty }{\underline{\lim }}\frac{%
\log ^{\left[ p-1\right] }T_{f\circ g}\left( r\right) }{T_{G}(r)}\leqslant
3\beta \cdot \rho _{f}(p,q).2^{^{\lambda _{g}}}\text{ if }q\geq l-1=1,
\end{equation*}%
where $\beta >1.$
\end{theorem}

\begin{proof}
As $\varepsilon \left( >0\right) $ is arbitrary and $T_{g}\left( r\right)
\leqslant \log ^{+}M_{g}\left( r\right) $ \{cf. \cite{r3} \}, we have from
Lemma \ref{l1} for all sufficiently large values of $r$ that%
\begin{equation}
\log ^{\left[ p-1\right] }T_{f\circ g}\left( r\right) \leqslant (\rho
_{f}(p,q)+\varepsilon )\log ^{\left[ q\right] }M_{g}\left( r\right) +O(1)~.
\label{11}
\end{equation}%
\textbf{Case I. }Let $q\geq l-1.$ Then from $\left( \ref{11}\right) $ we
obtain for all sufficiently large values of $r$ that%
\begin{equation}
\log ^{\left[ p-1\right] }T_{f\circ g}\left( r\right) \leqslant (\rho
_{f}(p,q)+\varepsilon )\log ^{\left[ l-1\right] }M_{g}\left( r\right) +O(1).
\notag
\end{equation}%
Since $\varepsilon \left( >0\right) $ we get from above that%
\begin{equation}
\underset{r\rightarrow \infty }{\underline{\lim }}\frac{\log ^{\left[ p-1%
\right] }T_{f\circ g}\left( r\right) }{\log ^{\left[ l-2\right] }T_{G}(r)}%
\leqslant \rho _{f}(p,q)\cdot \underset{r\rightarrow \infty }{\underline{%
\lim }}\frac{\log ^{\left[ l-1\right] }M_{g}\left( r\right) }{\log ^{\left[
l-2\right] }T_{G}(r)}~.  \label{13}
\end{equation}%
\textbf{Case II. }Let $q<l-1.$ Then from $\left( \ref{11}\right) $ we get
for all sufficiently large values of $r$ that%
\begin{eqnarray*}
\log ^{\left[ l-q-1\right] }\log ^{\left[ p-1\right] }T_{f\circ g}\left(
r\right) &\leqslant &\log ^{\left[ l-q-1\right] }\left\{ (\rho
_{f}(p,q)+\varepsilon )\log ^{\left[ q\right] }M_{g}\left( r\right)
+O(1)\right\} \\
i.e.,~\ \ \log ^{\left[ p+l-q-2\right] }T_{f\circ g}\left( r\right)
&\leqslant &\log ^{\left[ l-1\right] }M_{g}\left( r\right) +O(1) \\
i.e.,~\ \ \ \frac{\log ^{\left[ p+l-q-2\right] }T_{f\circ g}\left( r\right) 
}{\log ^{\left[ l-2\right] }T_{G}(r)} &\leqslant &\frac{\log ^{\left[ l-1%
\right] }M_{g}\left( r\right) +O(1)}{\log ^{\left[ l-2\right] }T_{G}(r)}~.
\end{eqnarray*}%
Therefore we get from above that%
\begin{equation}
\underset{r\rightarrow \infty }{\underline{\lim }}\frac{\log ^{\left[ p+l-q-2%
\right] }T_{f\circ g}\left( r\right) }{\log ^{\left[ l-2\right] }T_{G}(r)}%
\leqslant \underset{r\rightarrow \infty }{\underline{\lim }}\frac{\log ^{%
\left[ l-1\right] }M_{g}\left( r\right) }{\log ^{\left[ l-2\right] }T_{G}(r)}%
~.  \label{14}
\end{equation}

Now let $l>2.$ Since $\underset{r\rightarrow \infty }{\underline{\lim }}%
\frac{\log ^{\left[ l-2\right] }T_{g}(r)}{r^{\lambda _{g}^{\left[ l\right]
}\left( r\right) }}=1,$ for given $\varepsilon \left( 0<\varepsilon
<1\right) $ we get for a sequence of values of $r$ tending to infinity that%
\begin{equation*}
\log ^{\left[ l-2\right] }T_{g}(r)<(1+\varepsilon )r^{\lambda _{g}^{\left[ l%
\right] }\left( r\right) }
\end{equation*}%
and for all sufficiently large values of $r,$%
\begin{equation*}
\log ^{\left[ l-2\right] }T_{g}(r)>(1-\varepsilon )r^{\lambda _{g}^{\left[ l%
\right] }\left( r\right) }.
\end{equation*}

Since $\log M_{g}(r)\leq 3T_{g}(2r)$ \{cf. \cite{r3} \} and $T_{g}\left(
r\right) =O\left\{ T_{G}\left( r\right) \right\} $ as $r\rightarrow \infty $
\{cf. \cite{y} \}$,$ we get for a sequence of values of $r$ tending to
infinity and for any $\delta \left( >0\right) $ that%
\begin{eqnarray*}
\frac{\log ^{\left[ l-1\right] }M_{g}\left( r\right) }{\log ^{\left[ l-2%
\right] }T_{G}(r)} &<&\frac{\log ^{\left[ l-1\right] }M_{g}\left( r\right) }{%
\log ^{\left[ l-2\right] }T_{g}(r)+O(1)}<\frac{\log ^{\left[ l-2\right]
}T_{g}(2r)+O(1)}{\log ^{\left[ l-2\right] }T_{g}(r)+O(1)} \\
&<&\frac{(1+\varepsilon )}{(1-\varepsilon )}\cdot \frac{\left( 2r\right)
^{\lambda _{g}^{\left[ l\right] }+\delta }}{\left( 2r\right) ^{\lambda _{g}^{%
\left[ l\right] }+\delta -\lambda _{g}^{\left[ l\right] }\left( 2r\right) }}%
\cdot \frac{1}{r^{\lambda _{g}^{\left[ l\right] }\left( r\right) }}+O(1) \\
&<&\frac{(1+\varepsilon )}{(1-\varepsilon )}\cdot 2^{^{\lambda _{g}^{\left[ l%
\right] }+\delta }}+O(1)
\end{eqnarray*}%
because $r^{\lambda _{g}^{\left[ l\right] }+\delta -\lambda _{g}^{\left[ l%
\right] }\left( r\right) }$ is ultimately an increasing function of $r$ by
Lemma \ref{l5}.

Since $\varepsilon \left( >0\right) $ and $\delta \left( >0\right) $ are
both arbitrary, we get from above that%
\begin{equation}
\underset{r\rightarrow \infty }{\underline{\lim }}\frac{\log ^{\left[ l-1%
\right] }M_{g}\left( r\right) }{\log ^{\left[ l-2\right] }T_{G}(r)}\leq
2^{^{\lambda _{g}^{\left[ l\right] }}}~.  \label{15}
\end{equation}

Again let $l=2.$ Since $\underset{r\rightarrow \infty }{\underline{\lim }}%
\frac{T_{g}(r)}{r^{\lambda _{g}\left( r\right) }}=1,$ in view of condition
(v) of Definition \ref{d2} it follows for a sequence of values of $r$
tending to infinity and for a given $\varepsilon \left( 0<\varepsilon
<1\right) $ that%
\begin{equation*}
T_{g}(r)<(1+\varepsilon )r^{\lambda _{g}\left( r\right) }
\end{equation*}%
and for all large positive numbers of $r,$%
\begin{equation*}
T_{g}(r)>(1-\varepsilon )r^{\lambda _{g}\left( r\right) }.
\end{equation*}%
As $\log M_{g}(r)\leq 3T_{g}(2r)$ \{cf. \cite{r3} \} and $T_{g}\left(
r\right) =O\left\{ T_{G}\left( r\right) \right\} $ as $r\rightarrow \infty $
\{cf. \cite{y} \}$,$ we get for any $\delta \left( >0\right) $, $\beta >1$
and for a sequence of values of $r$ tending to infinity that%
\begin{eqnarray}
\frac{\log M_{g}(r)}{T_{G}(r)} &<&\beta \cdot \frac{\log M_{g}(r)}{T_{g}(r)}%
<\beta \cdot \frac{3(1+\varepsilon )}{(1-\varepsilon )}\cdot \frac{\left(
2r\right) ^{\lambda _{g}+\delta }}{\left( 2r\right) ^{\lambda _{g}+\delta
-\lambda _{g}\left( 2r\right) }}\cdot \frac{1}{r^{\lambda _{g}\left(
r\right) }}+O(1)  \notag \\
i.e.,~\frac{\log M_{g}(r)}{T_{G}(r)} &<&\frac{3\beta (1+\varepsilon )}{%
(1-\varepsilon )}\cdot 2^{^{\lambda _{g}+\delta }}+O(1)~.  \label{16}
\end{eqnarray}%
because $r^{\lambda _{g}+\delta -\lambda _{g}\left( r\right) }$ is
ultimately an increasing function of $r$ by Lemma \ref{l5}. Since $%
\varepsilon \left( >0\right) $ and $\delta \left( >0\right) $ are both
arbitrary, we get from $\left( \ref{16}\right) $ that%
\begin{equation}
\underset{r\rightarrow \infty }{\underline{\lim }}\frac{\log M_{g}(r)}{%
T_{G}(r)}\leq 3\beta \cdot 2^{^{\lambda _{g}}}~.  \label{17}
\end{equation}%
Therefore from $\left( \ref{13}\right) $ of Case I and $\left( \ref{15}%
\right) $ it follows that%
\begin{equation*}
\underset{r\rightarrow \infty }{\underline{\lim }}\frac{\log ^{\left[ p-1%
\right] }T_{f\circ g}\left( r\right) }{\log ^{\left[ l-2\right] }T_{G}(r)}%
\leqslant \rho _{f}(p,q)\cdot 2^{^{\lambda _{g}^{\left[ l\right] }}}~.
\end{equation*}%
This proves the first part of the theorem. Also from $\left( \ref{14}\right) 
$ of Case II and $\left( \ref{15}\right) $ we obtain that%
\begin{equation*}
\underset{r\rightarrow \infty }{\underline{\lim }}\frac{\log ^{\left[ p+l-q-2%
\right] }T_{f\circ g}\left( r\right) }{\log ^{\left[ l-2\right] }T_{G}(r)}%
\leqslant 2^{^{\lambda _{g}^{\left[ l\right] }}}~.
\end{equation*}%
Thus the second part of the theorem is established. Again putting $l=2$ in $%
\left( \ref{13}\right) $ of Case I and in view of $\left( \ref{17}\right) $
we obtain that%
\begin{equation*}
\underset{r\rightarrow \infty }{\underline{\lim }}\frac{\log ^{\left[ p-1%
\right] }T_{f\circ g}\left( r\right) }{T_{G}(r)}\leqslant 3\beta \cdot \rho
_{f}(p,q)\cdot 2^{^{\lambda _{g}}}~.
\end{equation*}%
Thus the third part of the theorem follows.
\end{proof}

\begin{corollary}
\label{Cor1} Under the same conditions of Theorem \ref{t5}, if $l=2$ then%
\begin{equation*}
\underset{r\rightarrow \infty }{\underline{\lim }}\frac{\log ^{\left[ p%
\right] }T_{f\circ g}\left( r\right) }{\log ^{\left[ q\right] }T_{G}(r)}\leq
1~.
\end{equation*}
\end{corollary}

\begin{proof}
If $q\geq 1$, then from $\left( \ref{11}\right) $ we obtain for all
sufficiently large values of $r$ that 
\begin{equation}
\log ^{\left[ p\right] }T_{f\circ g}\left( r\right) \leqslant \log ^{\left[
q+1\right] }M_{g}\left( r\right) +O(1)~.  \label{18}
\end{equation}%
Now from $\left( \ref{16}\right) $ we have for a sequence of values of $r$
tending to infinity that%
\begin{eqnarray}
\log M_{g}(r) &\leq &\left\{ \frac{3\beta (1+\varepsilon )}{(1-\varepsilon )}%
\cdot 2^{^{\lambda _{g}+\delta }}\right\} \cdot T_{G}(r)  \notag \\
i.e.,~\log ^{\left[ q+1\right] }M_{g}(r) &\leq &\log ^{\left[ q\right]
}T_{G}(r)+O(1)~.  \label{20}
\end{eqnarray}%
Now combining $\left( \ref{18}\right) $ and $\left( \ref{20}\right) $ it
follows for a sequence of values of $r$ tending to infinity that%
\begin{eqnarray*}
\log ^{\left[ p\right] }T_{f\circ g}\left( r\right) &\leqslant &\log ^{\left[
q\right] }T_{G}(r)+O(1) \\
i.e.,~\frac{\log ^{\left[ p\right] }T_{f\circ g}\left( r\right) }{\log ^{%
\left[ q\right] }T_{G}(r)} &\leq &1+\frac{O(1)}{\log ^{\left[ q\right]
}T_{G}(r)}~.
\end{eqnarray*}%
So from above we obtain that%
\begin{equation*}
\underset{r\rightarrow \infty }{\underline{\lim }}\frac{\log ^{\left[ p%
\right] }T_{f\circ g}\left( r\right) }{\log ^{\left[ q\right] }T_{G}(r)}\leq
1~.
\end{equation*}%
Thus the corollary follows.
\end{proof}

\begin{theorem}
\label{t6} Let $f$\ be meromorphic and $g$ be a transcendental entire such
that $\rho _{f}(p,q)$ and $\rho _{g}^{\left[ l\right] }$ are finite where $%
p,q,l$ are positive integers with $p>q$ and $l\geq 2$. Also let $G=g^{\alpha
}Q\left[ g\right] $ where $Q\left[ g\right] $ is a differential polynomial
in $g$, then for any $\alpha \geq 1$%
\begin{eqnarray*}
\left( i\right) ~\ \ \ \ \ ~\underset{r\rightarrow \infty }{\underline{\lim }%
}\frac{\log ^{\left[ p-1\right] }T_{f\circ g}\left( r\right) }{\log ^{\left[
l-2\right] }T_{G}(r)} &\leqslant &\rho _{f}(p,q)\cdot 2^{^{\rho _{g}^{\left[
l\right] }}}\text{ if }q\geq l-1>1, \\
\left( ii\right) ~\underset{r\rightarrow \infty }{\underline{\lim }}\frac{%
\log ^{\left[ p+l-q-2\right] }T_{f\circ g}\left( r\right) }{\log ^{\left[ l-2%
\right] }T_{G}(r)} &\leqslant &2^{^{\rho _{g}^{\left[ l\right] }}}\text{ if }%
q<l-1,
\end{eqnarray*}%
and%
\begin{equation*}
\left( iii\right) ~\underset{r\rightarrow \infty }{\underline{\lim }}\frac{%
\log ^{\left[ p-1\right] }T_{f\circ g}\left( r\right) }{T_{G}(r)}\leqslant
3\beta \cdot \rho _{f}(p,q)\cdot 2^{^{\rho _{g}}}\text{ if }q\geq l-1=1,
\end{equation*}%
where $\beta >1.$
\end{theorem}

\begin{proof}
\textbf{Case I. }Let $l>2.$ As $\underset{r\rightarrow \infty }{\overline{%
\lim }}\frac{\log ^{\left[ l-2\right] }T_{g}(r)}{r^{\rho _{g}^{\left[ l%
\right] }\left( r\right) }}=1$, for given $\varepsilon \left( 0<\varepsilon
<1\right) $ we obtain for all sufficiently large values of $r$ that%
\begin{equation*}
\log ^{\left[ l-2\right] }T_{g}(r)<(1+\varepsilon )r^{\rho _{g}^{\left[ l%
\right] }\left( r\right) }
\end{equation*}%
and for a sequence of values of $r$ tending to infinity,%
\begin{equation*}
\log ^{\left[ l-2\right] }T_{g}(r)>(1-\varepsilon )r^{\rho _{g}^{\left[ l%
\right] }\left( r\right) }~.
\end{equation*}%
Since $\log M_{g}(r)\leq 3T_{g}(2r)$ \{cf. \cite{r3} \} and $T_{g}\left(
r\right) =O\left\{ T_{G}\left( r\right) \right\} $ as $r\rightarrow \infty $
\{cf. \cite{y} \}$,$ for a sequence of values of $r$ tending to infinity we
get for any $\delta \left( >0\right) $ that%
\begin{eqnarray*}
\frac{\log ^{\left[ l-1\right] }M_{g}(r)}{\log ^{\left[ l-2\right] }T_{G}(r)}
&<&\frac{\log ^{\left[ l-1\right] }M_{g}(r)}{\log ^{\left[ l-2\right]
}T_{g}(r)+O(1)}<\frac{\log ^{\left[ l-2\right] }T_{g}(2r)+O(1)}{\log ^{\left[
l-2\right] }T_{g}(r)+O(1)} \\
&<&\frac{(1+\varepsilon )}{(1-\varepsilon )}\cdot \frac{\left( 2r\right)
^{\rho _{g}^{\left[ l\right] }+\delta }}{\left( 2r\right) ^{\rho _{g}^{\left[
l\right] }+\delta -\rho _{g}^{\left[ l\right] }\left( 2r\right) }}\cdot 
\frac{1}{r^{\rho _{g}^{\left[ l\right] }\left( r\right) }}+O(1) \\
&<&\frac{(1+\varepsilon )}{(1-\varepsilon )}\cdot 2^{^{\rho _{g}^{\left[ l%
\right] }+\delta }}
\end{eqnarray*}%
because $r^{\rho _{g}^{\left[ l\right] }+\delta -\rho _{g}^{\left[ l\right]
}\left( r\right) }$ is ultimately an increasing function of $r$ by Lemma \ref%
{l6}.

Since $\varepsilon \left( >0\right) $ and $\delta \left( >0\right) $ are
both arbitrary, we get from above that%
\begin{equation}
\underset{r\rightarrow \infty }{\underline{\lim }}\frac{\log ^{\left[ l-1%
\right] }M_{g}(r)}{\log ^{\left[ l-2\right] }T_{G}(r)}\leq 2^{^{\rho _{g}^{%
\left[ l\right] }}}~.  \label{21}
\end{equation}%
\textbf{Case II. }Let $l=2.$ Since $\underset{r\rightarrow \infty }{%
\overline{\lim }}\frac{T_{g}(r)}{r^{\rho _{g}\left( r\right) }}=1,$ in view
of condition (v) of Definition \ref{d3} it follows for all sufficiently
large values of $r$ and for a given $\varepsilon \left( 0<\varepsilon
<1\right) $ that%
\begin{equation*}
T_{g}(r)<(1+\varepsilon )r^{\rho _{g}\left( r\right) }
\end{equation*}%
and for a sequence of values of $r$ tending to infinity%
\begin{equation*}
T_{g}(r)>(1-\varepsilon )r^{\rho _{g}\left( r\right) }.
\end{equation*}%
As $\log M_{g}(r)\leq 3T_{g}(2r)$ \{cf. \cite{r3} \} and $T_{g}\left(
r\right) =O\left\{ T_{G}\left( r\right) \right\} $ as $r\rightarrow \infty $
\{cf. \cite{y} \}$,$ we get for any $\delta \left( >0\right) $, $\beta >1$
and for a sequence of values of $r$ tending to infinity that%
\begin{eqnarray*}
\frac{\log M_{g}(r)}{T_{G}(r)} &<&\beta \cdot \frac{\log M_{g}(r)}{T_{g}(r)}%
<\beta \cdot \frac{3(1+\varepsilon )}{(1-\varepsilon )}\cdot \frac{\left(
2r\right) ^{\rho _{g}+\delta }}{\left( 2r\right) ^{\rho _{g}+\delta -\rho
_{g}\left( 2r\right) }}\cdot \frac{1}{r^{\rho _{g}\left( r\right) }}+O(1) \\
i.e.,~\frac{\log M_{g}(r)}{T_{G}(r)} &<&\frac{3\beta (1+\varepsilon )}{%
(1-\varepsilon )}\cdot 2^{^{\rho _{g}+\delta }}+O(1)~.
\end{eqnarray*}%
because $r^{\rho _{g}+\delta -\rho _{g}\left( r\right) }$ is ultimately an
increasing function of $r$ by Lemma \ref{l6}.

Since $\varepsilon \left( >0\right) $ and $\delta \left( >0\right) $ are
both arbitrary, we get from the above that%
\begin{equation}
\underset{r\rightarrow \infty }{\underline{\lim }}\frac{\log M_{g}(r)}{%
T_{G}(r)}\leq 3\cdot 2^{^{\rho _{g}}}~.  \label{22}
\end{equation}%
Therefore from $\left( \ref{13}\right) $ and $\left( \ref{21}\right) $ it
follows that%
\begin{equation*}
\underset{r\rightarrow \infty }{\underline{\lim }}\frac{\log ^{\left[ p-1%
\right] }T_{f\circ g}\left( r\right) }{\log ^{\left[ l-2\right] }T_{G}(r)}%
\leqslant \rho _{f}(p,q)\cdot 2^{^{\rho _{g}^{\left[ l\right] }}}~.
\end{equation*}%
This proves the first part of the theorem. Similarly from $\left( \ref{14}%
\right) $ and $\left( \ref{21}\right) $ we get%
\begin{equation*}
\underset{r\rightarrow \infty }{\underline{\lim }}\frac{\log ^{\left[ p+l-q-2%
\right] }T_{f\circ g}\left( r\right) }{\log ^{\left[ l-2\right] }T_{G}(r)}%
\leqslant 2^{^{\rho _{g}^{\left[ l\right] }}}~.
\end{equation*}%
Thus the second part of the theorem follows.

Again putting $l=2$ in $\left( \ref{13}\right) $ and in view of $\left( \ref%
{22}\right) $ we obtain that%
\begin{equation*}
\underset{r\rightarrow \infty }{\underline{\lim }}\frac{\log ^{\left[ p-1%
\right] }T_{f\circ g}\left( r\right) }{T_{G}(r)}\leqslant 3\beta \cdot \rho
_{f}(p,q)\cdot 2^{^{\rho _{g}}}.
\end{equation*}%
Thus the third part of the theorem is established.
\end{proof}

\end{document}